\documentclass{article}
\usepackage{amssymb}
\usepackage{hyperref}

\newtheorem{theorem}{Theorem}

\newtheorem{definition}[theorem]{Definition}
\newtheorem{example}[theorem]{Example}

\newtheorem{proposition}[theorem]{Proposition}
\newtheorem{remark}[theorem]{Remark}

\newenvironment{proof}[1][Proof]{\noindent\textbf{#1.} }{\ \rule{0.5em}{0.5em}}
\input{tcilatex}
\numberwithin{theorem}{section}
\numberwithin{equation}{section}

\begin{document}

\title{Local Bianchi Identities in the Relativistic Non-Autonomous Lagrange
Geometry}
\author{Mircea Neagu}
\date{}
\maketitle

\begin{abstract}
The aim of this paper is to describe the local Bianchi identities for an $h$%
-normal $\Gamma $-linear connection of Cartan type $\nabla \Gamma $ on the
first-order jet space $J^{1}(\mathbb{R},M)$. In this direction, we present
the local expressions of the adapted components of the torsion and curvature
d-tensors produced by $\nabla \Gamma $ and we give the general local
expressions of Bianchi identities which connect these d-torsions and
d-curvatures.
\end{abstract}



\section{Introduction}

\hspace{5mm}It is well known that the 1-jet spaces are basic objects in the
study of classical and quantum field theories. For such a reason, a lot of
authors (Asanov \cite{Asanov}, Krupkov\'{a} \cite{Krupkova}, Saunders \cite%
{Saunders}, Vondra \cite{Vondra} and many others) studied the differential
geometry of 1-jet spaces. Going on the geometrical studies initiated by
Asanov \cite{Asanov} and using as a pattern the Lagrangian geometrical ideas
developed by Miron, Anastasiei and Buc\u{a}taru in the monographs \cite%
{Mir-An} and \cite{Buc-Mir}, the author of this paper has recently developed
the \textit{Riemann-Lagrange geometry of the }$1$\textit{-jet spaces} \cite%
{Neagu Carte}. This is a general geometrical framework for the study of 
\textit{\textbf{\textit{relativistic n}on-autonomous} (\textbf{rheonomic} }or%
\textit{\ \textbf{time dependent}) \textbf{Lagrangians}} \cite{Neagu-Rheon}
or \textit{relativistic multi-time dependent Lagrangians} \cite{Neagu Carte}.

We underline that a \textit{\textbf{\textit{classical n}on-autonomous} (%
\textbf{rheonomic} }or\textit{\ \textbf{time dependent}) \textbf{Lagrangian
geometry}}, that is a geometrization of Lagrangians depending on an \textit{%
\textbf{absolute time}}, was sketched at level of ideas by Miron and
Anastasiei in the last chapter of the book \cite{Mir-An}. That classical
non-autonomous Lagrangian geometry was developed further by Anastasiei and
Kawaguchi \cite{Anastas-Kaw} or Frigioiu \cite{Frigioiu}.

In what follows, we try to expose the main geometrical and physical aspects
which differentiate the both geometrical theories: the \textit{\textbf{%
\textit{jet relativistic n}on-autonomous} \textbf{Lagrangian geometry}} \cite%
{Neagu-Rheon} and the \textit{\textbf{\textit{classical n}on-\-au\-to\-nomous%
} \textbf{Lagrangian geometry }}\cite{Mir-An}.

In this way, we point out that the \textit{\textbf{\textit{relativistic n}%
on-autonomous} \textbf{Lagrangian geometry}} \cite{Neagu-Rheon} has as
natural house the 1-jet space $J^{1}(\mathbb{R},M)$, where $\mathbb{R}$ is
the manifold of real numbers having the coordinate $t$. This coordinate
re\-pre\-sents for us the usual \textit{relativistic time}. We recall that
the 1-jet space $J^{1}(\mathbb{R},M)$ is regarded as a vector bundle over
the product manifold $\mathbb{R}\times M$, having the fibre type $\mathbb{R}%
^{n}$, where $n$ is the dimension of the \textit{spatial} manifold $M$. In
mechanical terms, if the manifold $M$ has the spatial local coordinates $%
(x^{i})_{i=\overline{1,n}}$, then the 1-jet vector bundle $J^{1}(\mathbb{R}%
,M)\rightarrow \mathbb{R}\times M$ can be regarded as a \textit{bundle of
configurations} having the local coordinates $(t,x^{i},y_{1}^{i})$; these
transform by the rules \cite{Neagu-Iasi}%
\begin{equation}
\left\{ 
\begin{array}{l}
\widetilde{t}=\widetilde{t}(t)\medskip \\ 
\widetilde{x}^{i}=\widetilde{x}^{i}(x^{j})\medskip \\ 
\widetilde{y}_{1}^{i}=\dfrac{\partial \widetilde{x}^{i}}{\partial x^{j}}%
\dfrac{dt}{d\widetilde{t}}\cdot y_{1}^{j}.%
\end{array}%
\right.  \label{rgg}
\end{equation}

\begin{remark}
The form of the jet transformation group (\ref{rgg}) stands out by the 
\textbf{relativistic }\textit{character} of the \textbf{time} $t$.
\end{remark}

Comparatively, the \textit{\textbf{\textit{classical n}on-autonomous} 
\textbf{Lagrangian geometry }}from \cite{Mir-An} has as \textit{bundle of
configurations }the vector bundle $\mathbb{R}\times TM\rightarrow M,$ whose
local coordinates $(t,x^{i},y^{i})$ transform by the rules 
\begin{equation}
\left\{ 
\begin{array}{l}
\widetilde{t}=t\medskip \\ 
\widetilde{x}^{i}=\widetilde{x}^{i}(x^{j})\medskip \\ 
\widetilde{y}^{i}=\dfrac{\partial \widetilde{x}^{i}}{\partial x^{j}}\cdot
y^{j},%
\end{array}%
\right.  \label{agg}
\end{equation}%
where $TM$ is the tangent bundle of the spatial manifold $M$.

\begin{remark}
The form of the transformation group (\ref{agg}) stands out by the \textbf{%
absolute }\textit{character} of the \textbf{time} $t$.
\end{remark}

It is obvious that the jet transformation group (\ref{rgg}) from the \textit{%
relativistic non-autonomous Lagrangian geometry} is more general and more
natural than the transformation group (\ref{agg}) used in the \textit{%
classical non-autonomous Lagrangian geometry}. The naturalness of the
transformation group (\ref{rgg}) comes from the fact that the relativity of
time is well-known. In contrast, the transformation group (\ref{agg})
ignores the temporal reparametrizations, emphasizing in this way the
absolute character of the usual time coordinate $t$.

From a geometrical point of view, we point out that the whole \textit{%
classical non-autonomous Lagrangian geometry} initiated by Miron and
Anastasiei \cite{Mir-An} relies on the study of the \textit{absolute energy
action functional}%
\[
\mathbb{E}_{1}(c)=\int_{a}^{b}L(t,x^{i},y^{i})dt, 
\]%
where $L:\mathbb{R}\times TM\rightarrow \mathbb{R}$ is a Lagrangian function
and $y^{i}=dx^{i}/dt,$ whose Euler-Lagrange equations produce a semispray $%
G^{i}(t,x^{k},y^{k})$ and a corresponding nonlinear connection%
\[
N_{j}^{i}=\frac{\partial G^{i}}{\partial y^{j}}. 
\]%
Therefore, the authors construct the adapted bases of vector and covector
fields, together with the adapted components of the $N$-linear connections
and their corresponding d-torsions and d-curvatures. But, because $%
L(t,x^{i},y^{i})$ is a real function, we deduce that the previous
geometrical theory has the following impediment: -\textit{the energy action
functional depends on the reparametrizations }$t\longleftrightarrow 
\widetilde{t}(t)$\textit{\ of the same curve }$c.$ Thus, in order to avoid
this inconvenience, the Finsler case imposes the 1-positive homogeneity
condition \cite{Bao}%
\[
L(t,x^{i},\lambda y^{i})=\lambda L(t,x^{i},y^{i}),\text{ }\forall \text{ }%
\lambda >0. 
\]

Alternatively, the \textit{relativistic non-autonomous Lagrangian geometry}
from \cite{Neagu-Rheon} uses the \textit{relativistic energy action
functional}%
\[
\mathbb{E}_{2}(c)=\int_{a}^{b}L(t,x^{i},y_{1}^{i})\sqrt{h_{11}(t)}dt, 
\]%
where $L:J^{1}(\mathbb{R},M)\rightarrow \mathbb{R}$ is a jet Lagrangian
function and $h_{11}(t)$ is a Riemannian metric on the relativistic time
manifold $\mathbb{R}$. This functional is now independent by the
reparametrizations $t\longleftrightarrow \widetilde{t}(t)$ of the same curve 
$c$, even if $L$ is only a function. The Euler-Lagrange equations of the 
\textit{Lagrangian}%
\[
\mathcal{L}=L(t,x^{i},y_{1}^{i})\sqrt{h_{11}(t)} 
\]%
produce a relativistic time dependent semispray \cite{Neagu-Rheon}%
\[
\mathcal{S=}\left( H_{(1)1}^{(i)},\text{ }G_{(1)1}^{(i)}\right) , 
\]%
which gives the jet nonlinear connection \cite{Neagu-Iasi}%
\[
\Gamma _{\mathcal{S}}=\left( M_{(1)1}^{(j)}=2H_{(1)1}^{(j)},\text{ }%
N_{(1)k}^{(j)}=\frac{\partial G_{(1)1}^{(j)}}{\partial y_{1}^{k}}\right) . 
\]%
With these geometrical tools we can construct in the relativistic
non-autonomous Lagrangian geometry the distinguished (d-) linear
connections, together with their d-torsions and d-curvatures, which
naturally generalize the similar geometrical objects from the classical
non-autonomous Lagrangian geometry \cite{Mir-An}.

In this respect, the author of this paper believes that the jet relativistic
geometrical approach proposed in the papers \cite{Neagu-Iasi}, \cite%
{Neagu-Rheon}, \cite{Neagu-Stoica} has more geometrical and physical
meanings than the theory proposed by Miron and Anastasiei in \cite{Mir-An}.
For such a reason, the aim of this paper is to describe the Bianchi
identities that govern the jet relativistic non-autonomous Lagrangian
geometry \cite{Neagu-Rheon}. These Bianchi identities are necessary for the
construction of the \textit{generalized Maxwell equations} that characterize
the electromagnetic theory in the jet relativistic non-autonomous Lagrangian
geometrical background.

In conclusion, in order to remark the main similitudes and differences
between these geometrical theories, we invite the reader to compare the 
\textit{classical }and \textit{relativistic non-autonomous Lagrangian
geometries} exposed in the works \cite{Mir-An} and \cite{Neagu-Rheon}.

As a final remark, we point out that for a lot of mathematicians (such as
Crampin \cite{Crampin}, Krupkov\'{a} \cite{Krupkova}, de L\'{e}on \cite{de
Leon}, Sarlet \cite{Sarlet} and others) the non-autonomous Lagrangian
geometry is constructed on the first jet bundle $J^{1}\pi $ of a fibered
manifold $\pi :M^{n+1}\longrightarrow \mathbb{R}.$ In their works, if $%
(t,x^{i})$ are the local coordinates on the $n+1$-dimensional manifold $M$
such that $t$ is a global coordinate for the fibers of the submersion $\pi $
and $x^{i}$ are transverse coordinates of the induced foliation, then a
change of coordinates on $M$ is given by%
\begin{equation}
\left\{ 
\begin{array}{ll}
\widetilde{t}=\widetilde{t}(t),\medskip & \dfrac{d\widetilde{t}}{dt}\neq 0
\\ 
\widetilde{x}^{i}=\widetilde{x}^{i}(x^{j},t), & \text{rank}\left( \dfrac{%
\partial \widetilde{x}^{i}}{\partial x^{j}}\right) =n.%
\end{array}%
\right.  \label{Krupkova}
\end{equation}

Altough the 1-jet extension of the transformation rules (\ref{Krupkova}) is
more general than the transformation group (\ref{rgg}), the author ot this
paper considers that the transformation group (\ref{rgg}) is nevertheless
more appropriate for his final purpose, namely: -- the development of a 
\textit{relativistic non-autonomous Lagrangian gravitational and
electromagnetic field theory} which to be characterized by some \textit{%
generalized Einstein and Maxwell equations} \cite{Neagu-Rheon}. In this
direction, we need the \textit{local Bianchi identities} that govern the jet
relativistic non-autonomous Lagrangian geometry \cite{Neagu-Iasi}, \cite%
{Neagu-Rheon}, \cite{Neagu-Stoica}.

\section{The adapted components of the jet $\Gamma $-linear connections}

\label{coli}

\hspace{5mm}Let us suppose that the 1-jet space $E=J^{1}(\mathbb{R},M)$ is
endowed with a nonlinear connection%
\begin{equation}
\Gamma =\left( M_{(1)1}^{(i)},N_{(1)j}^{(i)}\right) ,  \label{gen_nlc}
\end{equation}%
where the local coefficients $M_{(1)1}^{(i)}$ (resp., $N_{(1)j}^{(i)}$) are
called the \textit{temporal }(resp.,\textit{\ spatial})\textit{\ components
of }$\Gamma $. Note that the transformation rules of the local components of
the nonlinear connection are expressed by \cite{Neagu-Iasi}%
\begin{equation}
\begin{array}{l}
\widetilde{M}_{(1)1}^{(k)}=M_{(1)1}^{(j)}\left( \dfrac{dt}{d\widetilde{t}}%
\right) ^{2}\dfrac{\partial \widetilde{x}^{k}}{\partial x^{j}}-\dfrac{dt}{d%
\widetilde{t}}\dfrac{\partial \widetilde{y}_{1}^{k}}{\partial t},\medskip \\ 
\widetilde{N}_{(1)l}^{(k)}=N_{(1)i}^{(j)}\dfrac{dt}{d\widetilde{t}}\dfrac{%
\partial x^{i}}{\partial \widetilde{x}^{l}}\dfrac{\partial \widetilde{x}^{k}%
}{\partial x^{j}}-\dfrac{\partial x^{i}}{\partial \widetilde{x}^{l}}\dfrac{%
\partial \widetilde{y}_{1}^{k}}{\partial x^{i}}.%
\end{array}
\label{nlc_tr_rules}
\end{equation}

\begin{example}
Let us consider $h=(h_{11}(t))$ (resp., $\varphi =(\varphi _{ij}(x))$) a
Riemannian metric on the temporal manifold $\mathbb{R}$ (resp., the spatial
manifold $M$) and let 
\[
\varkappa _{11}^{1}=\frac{h^{11}}{2}\frac{dh_{11}}{dt},\qquad \gamma
_{jk}^{i}=\frac{\varphi ^{im}}{2}\left( \frac{\partial \varphi _{jm}}{%
\partial x^{k}}+\frac{\partial \varphi _{km}}{\partial x^{j}}-\frac{\partial
\varphi _{jk}}{\partial x^{m}}\right) , 
\]%
where $h^{11}=1/h_{11}$, be their Christoffel symbols. Then, the set of
local functions%
\[
\mathring{\Gamma}=\left( \mathring{M}_{(1)1}^{(j)},\mathring{N}%
_{(1)i}^{(j)}\right) , 
\]%
where%
\begin{equation}
\mathring{M}_{(1)1}^{(j)}=-\varkappa _{11}^{1}y_{1}^{j},\qquad \mathring{N}%
_{(1)i}^{(j)}=\gamma _{im}^{j}y_{1}^{m},  \label{can_nlc}
\end{equation}%
represents a nonlinear connection on the 1-jet space $J^{1}(\mathbb{R},M)$.
This jet nonlinear connection is called the \textbf{canonical nonlinear
connection attached to the pair of Riemannian metrics} $(h_{11}(t),\varphi
_{ij}(x))$.
\end{example}

In the sequel, starting with the fixed nonlinear connection $\Gamma $ given
by (\ref{gen_nlc}), we construct the \textit{horizontal} vector fields%
\begin{equation}
\frac{\delta }{\delta t}=\frac{\partial }{\partial t}-M_{(1)1}^{(j)}\frac{%
\partial }{\partial y_{1}^{j}},\qquad \frac{\delta }{\delta x^{i}}=\frac{%
\partial }{\partial x^{i}}-N_{(1)i}^{(j)}\frac{\partial }{\partial y_{1}^{j}}%
,  \label{v-a-b}
\end{equation}%
and the \textit{vertical} covector fields%
\begin{equation}
\delta y_{1}^{i}=dy_{1}^{i}+M_{(1)1}^{(i)}dt+N_{(1)j}^{(i)}dx^{j}.
\label{cv-a-b}
\end{equation}

It is easy to see now that the set of vector fields%
\begin{equation}
\left\{ \frac{\delta }{\delta t},\frac{\delta }{\delta x^{i}},\dfrac{%
\partial }{\partial y_{1}^{i}}\right\} \subset \mathcal{X}(E)
\label{ad-basis-vf}
\end{equation}%
represents a \textit{basis} in the set of vector fields on $J^{1}(\mathbb{R}%
,M)$, and the set of covector fields 
\begin{equation}
\left\{ dt,dx^{i},\delta y_{1}^{i}\right\} \subset \mathcal{X}^{\ast }(E)
\label{ad-basis-cvf}
\end{equation}%
represents its \textit{dual basis} in the set of 1-forms on $J^{1}(\mathbb{R}%
,M)$.

\begin{definition}
The dual bases (\ref{ad-basis-vf}) and (\ref{ad-basis-cvf}) are called the 
\textbf{adapted bases} attached to the nonlinear connection $\Gamma $ on the
1-jet space $E=J^{1}(\mathbb{R},M)$.
\end{definition}

\begin{remark}
It is important to note that the local transformation laws of the elements
of the adapted bases (\ref{ad-basis-vf}) and (\ref{ad-basis-cvf}) are 
\textit{classical tensorial ones}:%
\begin{equation}
\begin{array}{lll}
\dfrac{\delta }{\delta t}=\dfrac{d\widetilde{t}}{dt}\dfrac{\delta }{\delta 
\widetilde{t}}, & \dfrac{\delta }{\delta x^{i}}=\dfrac{\partial \widetilde{x}%
^{j}}{\partial x^{i}}\dfrac{\delta }{\delta \widetilde{x}^{j}}, & \dfrac{%
\partial }{\partial y_{1}^{i}}=\dfrac{\partial \widetilde{x}^{j}}{\partial
x^{i}}\dfrac{dt}{d\widetilde{t}}\dfrac{\partial }{\partial \widetilde{y}%
_{1}^{j}},\medskip \\ 
dt=\dfrac{dt}{d\widetilde{t}}d\widetilde{t}, & dx^{i}=\dfrac{\partial x^{i}}{%
\partial \widetilde{x}^{j}}d\widetilde{x}^{j}, & \delta y_{1}^{i}=\dfrac{%
\partial x^{i}}{\partial \widetilde{x}^{j}}\dfrac{d\widetilde{t}}{dt}\delta 
\widetilde{y}_{1}^{j}.\medskip%
\end{array}
\label{tr_rules_a_b}
\end{equation}%
For such a reason, is our choice to describe the geometrical objects of the
1-jet space $E=J^{1}(\mathbb{R},M)$ in local adapted components.
\end{remark}

It is obvious that the Lie algebra $\mathcal{X}(E)$ of the vector fields on $%
E=J^{1}(\mathbb{R},M)$ decomposes as $\mathcal{X}(E)=\mathcal{X}(\mathcal{H}%
_{\mathbb{R}})\oplus \mathcal{X}(\mathcal{H}_{M})\oplus \mathcal{X}(\mathcal{%
V}),$ where%
\[
\mathcal{X}(\mathcal{H}_{\mathbb{R}})=Span\left\{ {\frac{\delta }{\delta t}}%
\right\} ,\quad \mathcal{X}(\mathcal{H}_{M})=Span\left\{ {\frac{\delta }{%
\delta x^{i}}}\right\} ,\quad \mathcal{X}(\mathcal{V})=Span\left\{ {\frac{%
\partial }{\partial y_{1}^{i}}}\right\} . 
\]%
Let us denote as $h_{\mathbb{R}}$, $h_{M}$ (horizontal) and $v$ (vertical)
the canonical projections produced by the above decomposition.

\begin{definition}
A linear connection $\nabla :\mathcal{X}(E)\times \mathcal{X}(E)\rightarrow 
\mathcal{X}(E)$, which verifies the conditions $\nabla h_{\mathbb{R}}=0,$ $%
\nabla h_{M}=0$ and $\nabla v=0$, is called a $\Gamma $\textbf{-linear
connection} on the 1-jet space $E=J^{1}(\mathbb{R},M)$. Obviously, the local
description of a $\Gamma $-linear connection $\nabla $ on $E$ is given by a
set of \textbf{nine} local adapted components%
\begin{equation}
\nabla \Gamma =\left( \bar{G}_{11}^{1},\;G_{i1}^{k},\;G_{(1)(j)1}^{(i)(1)},\;%
\bar{L}_{1j}^{1},\;L_{ij}^{k},\;L_{(1)(j)k}^{(i)(1)},\;\bar{C}%
_{1(k)}^{1(1)},\;C_{i(k)}^{j(1)},\;C_{(1)(j)(k)}^{(i)(1)(1)}\right) ,
\label{comp_lin_conn}
\end{equation}%
which are produced by the relations:%
\[
\begin{array}{l}
(h_{\mathbb{R}})\hspace{4mm}{\nabla _{{\dfrac{\delta }{\delta t}}}{\dfrac{%
\delta }{\delta t}}=\bar{G}_{11}^{1}{\dfrac{\delta }{\delta t}},\;\;\nabla _{%
{\dfrac{\delta }{\delta t}}}{\dfrac{\delta }{\delta x^{i}}}=G_{i1}^{k}{%
\dfrac{\delta }{\delta x^{k}}},\;\;\nabla _{{\dfrac{\delta }{\delta t}}}{%
\dfrac{\partial }{\partial y_{1}^{i}}}=G_{(1)(i)1}^{(k)(1)}{\dfrac{\partial 
}{\partial y_{1}^{k}}}},\medskip \\ 
(h_{M})\;\;{\nabla _{{\dfrac{\delta }{\delta x^{j}}}}{\dfrac{\delta }{\delta
t}}=\bar{L}_{1j}^{1}{\dfrac{\delta }{\delta t}},\;\;\nabla _{{\dfrac{\delta 
}{\delta x^{j}}}}{\dfrac{\delta }{\delta x^{i}}}=L_{ij}^{k}{\dfrac{\delta }{%
\delta x^{k}}},\;\;\nabla _{{\dfrac{\delta }{\delta x^{j}}}}{\dfrac{\partial 
}{\partial y_{1}^{i}}}=L_{(1)(i)j}^{(k)(1)}{\dfrac{\partial }{\partial
y_{1}^{k}}}},\medskip \\ 
(v)\hspace{3mm}{\nabla _{{\dfrac{\partial }{\partial y_{1}^{j}}}}{\dfrac{%
\delta }{\delta t}}=\bar{C}_{1(j)}^{1(1)}{\dfrac{\delta }{\delta t}}%
,\;\nabla _{{\dfrac{\partial }{\partial y_{1}^{j}}}}{\dfrac{\delta }{\delta
x^{i}}}=C_{i(j)}^{k(1)}{\dfrac{\delta }{\delta x^{k}}},\;\nabla _{{\dfrac{%
\partial }{\partial y_{1}^{j}}}}{\dfrac{\partial }{\partial y_{1}^{i}}}%
=C_{(1)(i)(j)}^{(k)(1)(1)}{\dfrac{\partial }{\partial y_{1}^{k}}}}.%
\end{array}%
\]
\end{definition}

\begin{example}
Let us consider the nonlinear connection $\mathring{\Gamma}$ given by (\ref%
{can_nlc}), produced by the pair of Riemannian metrics $(h_{11}(t),\varphi
_{ij}(x))$. Then, the set of adapted local components%
\[
B\mathring{\Gamma}=\left( \bar{G}_{11}^{1},\;0,\;G_{(1)(i)1}^{(k)(1)},\;0,%
\;L_{ij}^{k},\;L_{(1)(i)j}^{(k)(1)},\;0,\;0,\;0\right) , 
\]%
where%
\[
\bar{G}_{11}^{1}=\varkappa _{11}^{1},\;\;G_{(1)(i)1}^{(k)(1)}=-\delta
_{i}^{k}\varkappa _{11}^{1},\;\;L_{ij}^{k}=\gamma
_{ij}^{k},\;\;L_{(1)(i)j}^{(k)(1)}=\gamma _{ij}^{k}, 
\]%
defines a $\mathring{\Gamma}$-linear connection on the 1-jet space $E=J^{1}(%
\mathbb{R},M)$. This connection is called the \textbf{Berwald linear
connection attached to the Riemannian metrics }$h_{11}(t)$ \textbf{and} $%
\varphi _{ij}(x)$.
\end{example}

Now, let $\nabla $ be a fixed $\Gamma $-linear connection on the 1-jet space 
$E=J^{1}(\mathbb{R},M)$, given by the adapted local components (\ref%
{comp_lin_conn}).

\begin{definition}
A geometrical object $D=\left( D_{1k(1)(l)...}^{1i(j)(1)...}\right) $ on the
1-jet vector bundle $E=J^{1}(\mathbb{R},M)$, whose local components
transform by the rules%
\[
D_{1k(1)(l)...}^{1i(j)(1)...}=\widetilde{D}_{1r(1)(s)...}^{1p(m)(1)...}\frac{%
dt}{d\widetilde{t}}\frac{\partial x^{i}}{\partial \widetilde{x}^{p}}\left( 
\frac{\partial x^{j}}{\partial \widetilde{x}^{m}}\frac{d\widetilde{t}}{dt}%
\right) \frac{d\widetilde{t}}{dt}\frac{\partial \widetilde{x}^{r}}{\partial
x^{k}}\left( \frac{\partial \widetilde{x}^{s}}{\partial x^{l}}\frac{dt}{d%
\widetilde{t}}\right) ..., 
\]%
is called a \textbf{d-tensor field}.
\end{definition}

\begin{example}
If $h_{11}(t)$ is a Riemannian metric on the time manifold $\mathbb{R}$,
then the geometrical object $\mathbf{J}=\left( \mathbf{J}_{(1)1j}^{(i)}%
\right) ,$ where $\mathbf{J}_{(1)1j}^{(i)}=h_{11}\delta _{j}^{i},$
represents a d-tensor field on the 1-jet space $E=J^{1}(\mathbb{R},M)$. This
is called the $h$\textbf{-normalization d-tensor field}.
\end{example}

The $\Gamma $-linear connection $\nabla $ naturally induces a linear
connection on the set of the d-tensors of the 1-jet vector bundle $E$, in
the following way: $-$ starting with $X\in \mathcal{X}(E)$ a vector field
and $D$ a d-tensor field on $E$, locally expressed by%
\[
\begin{array}{l}
\medskip {X=X^{1}{\dfrac{\delta }{\delta t}}+X^{r}{\dfrac{\delta }{\delta
x^{r}}}+X_{(1)}^{(r)}{\dfrac{\partial }{\partial y_{1}^{r}}},} \\ 
{D=D_{1k(1)(l)\ldots }^{1i(j)(1)\ldots }{\dfrac{\delta }{\delta t}}\otimes {%
\dfrac{\delta }{\delta x^{i}}}\otimes {\dfrac{\partial }{\partial y_{1}^{j}}}%
\otimes dt\otimes dx^{k}\otimes \delta y_{1}^{l}\ldots ,}%
\end{array}%
\]%
we introduce the covariant derivative%
\[
\begin{array}{l}
\medskip \nabla _{X}D=X^{1}\nabla _{\dfrac{\delta }{\delta t}}D+X^{p}\nabla
_{\dfrac{\delta }{\delta x^{p}}}D+X_{(1)}^{(p)}\nabla _{\dfrac{\partial }{%
\partial y_{1}^{p}}}D=\left\{ X^{1}D_{1k(1)(l)\ldots /1}^{1i(j)(1)\ldots
}+X^{p}\cdot \right. \\ 
\left. \cdot D_{1k(1)(l)\ldots |p}^{1i(j)(1)\ldots
}+X_{(1)}^{(p)}D_{1k(1)(l)\ldots }^{1i(j)(1)\ldots }|_{(p)}^{(1)}\right\} {{%
\dfrac{\delta }{\delta t}}\otimes {\dfrac{\delta }{\delta x^{i}}}\otimes {%
\dfrac{\partial }{\partial y_{1}^{j}}}\otimes dt\otimes dx^{k}\otimes \delta
y_{1}^{l}\ldots ,}%
\end{array}%
\]%
where

$(h_{\mathbb{R}})\hspace{6mm}\left\{ 
\begin{array}{l}
\medskip {D_{1k(1)(l)\ldots /1}^{1i(j)(1)\ldots }={\dfrac{\delta
D_{1k(1)(l)\ldots }^{1i(j)(1)\ldots }}{\delta t}}+D_{1k(1)(l)\ldots
}^{1i(j)(1)\ldots }\bar{G}_{11}^{1}+} \\ 
\medskip +D_{1k(1)(l)\ldots }^{1r(j)(1)\ldots }G_{r1}^{i}+D_{1k(1)(l)\ldots
}^{1i(r)(1)\ldots }G_{(1)(r)1}^{(j)(1)}+\ldots - \\ 
-D_{1k(1)(l)\ldots }^{1i(j)(1)\ldots }\bar{G}_{11}^{1}-D_{1r(1)(l)\ldots
}^{1i(j)(1)\ldots }G_{k1}^{r}-D_{1k(1)(r)\ldots }^{1i(j)(1)\ldots
}G_{(1)(l)1}^{(r)(1)}-\ldots ,%
\end{array}%
\right. \medskip$

$(h_{M})\hspace{5mm}\left\{ 
\begin{array}{l}
\medskip {D_{1k(1)(l)\ldots |p}^{1i(j)(1)\ldots }={\dfrac{\delta
D_{1k(1)(l)\ldots }^{1i(j)(1)\ldots }}{\delta x^{p}}}+D_{1k(1)(l)\ldots
}^{1i(j)(1)\ldots }\bar{L}_{1p}^{1}+} \\ 
\medskip +D_{1k(1)(l)\ldots }^{1r(j)(1)\ldots }L_{rp}^{i}+D_{1k(1)(l)\ldots
}^{1i(r)(1)\ldots }L_{(1)(r)p}^{(j)(1)}+\ldots - \\ 
-D_{1k(1)(l)\ldots }^{1i(j)(1)\ldots }\bar{L}_{1p}^{1}-D_{1r(1)(l)\ldots
}^{1i(j)(1)\ldots }L_{kp}^{r}-D_{1k(1)(r)\ldots }^{1i(j)(1)\ldots
}L_{(1)(l)p}^{(r)(1)}-\ldots ,%
\end{array}%
\right. \medskip$

$(v)\hspace{8mm}\left\{ 
\begin{array}{l}
\medskip {D_{1k(1)(l)\ldots }^{1i(j)(1)\ldots }|_{(p)}^{(1)}={\dfrac{%
\partial D_{1k(1)(l)\ldots }^{1i(j)(1)\ldots }}{\partial y_{1}^{p}}}%
+D_{1k(1)(l)\ldots }^{1i(j)(1)\ldots }\bar{C}_{1(p)}^{1(1)}+} \\ 
\medskip +D_{1k(1)(l)\ldots }^{1r(j)(1)\ldots
}C_{r(p)}^{i(1)}+D_{1k(1)(l)\ldots }^{1i(r)(1)\ldots
}C_{(1)(r)(p)}^{(j)(1)(1)}+\ldots - \\ 
-D_{1k(1)(l)\ldots }^{1i(j)(1)\ldots }\bar{C}_{1(p)}^{1(1)}-D_{1r(1)(l)%
\ldots }^{1i(j)(1)\ldots }C_{k(p)}^{r(1)}-D_{1k(1)(r)\ldots
}^{1i(j)(1)\ldots }C_{(1)(l)(p)}^{(r)(1)(1)}-\ldots .%
\end{array}%
\right. $

\begin{definition}
The local derivative operators \textquotedblright $_{/1}$\textquotedblright
, \textquotedblright $_{|p}$\textquotedblright\ and \textquotedblright $%
|_{(p)}^{(1)}$\textquotedblright\ are called the $\mathbb{R}$\textbf{%
-horizontal}\textit{, }$M$\textbf{-horizontal }and \textbf{vertical
covariant derivatives produced by the }$\Gamma $\textbf{-linear connection }$%
\nabla \Gamma $.
\end{definition}

\section{$h$-Normal $\Gamma $-linear connections}

\hspace{5mm}The big number (nine) of components that characterize a general $%
\Gamma $-linear connection $\nabla $ on the 1-jet space $E=J^{1}(\mathbb{R}%
,M)$ determines us to consider the following geometrical concept:

\begin{definition}
A $\Gamma $-linear connection $\nabla $ on $E$, whose local components (\ref%
{comp_lin_conn}) verify the relations 
\[
\bar{G}_{11}^{1}=\varkappa _{11}^{1},\quad \bar{L}_{1j}^{1}=0,\quad \bar{C}%
_{1(k)}^{1(1)}=0,\quad \nabla \mathbf{J}=0, 
\]%
where $h=(h_{11}(t))$ is a Riemannian metric on $\mathbb{R}$, $\varkappa
_{11}^{1}$ is its Christoffel symbol and $\mathbf{J}$ is the $h$%
-normalization d-tensor field, is called an $h$\textbf{-normal }$\Gamma $%
\textbf{-linear connection} on $E$.
\end{definition}

\begin{remark}
The condition $\nabla \mathbf{J}=0$ is equivalent with the local equalities%
\[
\mathbf{J}_{(1)1j/1}^{(i)}=0,\quad \mathbf{J}_{(1)1j|k}^{(i)}=0,\quad 
\mathbf{J}_{(1)1j}^{(i)}|_{(k)}^{(1)}=0, 
\]%
where \textquotedblright $_{/1}$\textquotedblright , \textquotedblright $%
_{|k}$\textquotedblright\ and \textquotedblright $|_{(k)}^{(1)}$%
\textquotedblright\ represent the $\mathbb{R}$-horizontal, $M$-horizontal
and vertical local covariant derivatives produced by the $\Gamma $-linear
connection $\nabla \Gamma $.
\end{remark}

In this context, we can prove the following important local geometrical
result:

\begin{theorem}
\label{TH-h-normal} The components of an $h$-normal $\Gamma $-linear
connection $\nabla $ verify the identities:%
\begin{equation}
\begin{array}{lll}
\medskip \bar{G}_{11}^{1}=\varkappa _{11}^{1}, & \bar{L}_{1j}^{1}=0, & \bar{C%
}_{1(k)}^{1(1)}=0, \\ 
\medskip G_{(1)(i)1}^{(k)(1)}=G_{i1}^{k}-\delta _{i}^{k}\varkappa _{11}^{1},
& L_{(1)(i)j}^{(k)(1)}=L_{ij}^{k}, & 
C_{(1)(i)(j)}^{(k)(1)(1)}=C_{i(j)}^{k(1)}.%
\end{array}
\label{h-normal-relations}
\end{equation}
\end{theorem}

\begin{proof}
The first three relations from (\ref{h-normal-relations}) are a direct
consequence of the definition of an $h$-normal $\Gamma $-linear connection $%
\nabla $.

The condition $\nabla \mathbf{J}=0$ implies the local relations%
\[
{h_{11}G_{(1)(j)1}^{(i)(1)}=h_{11}G_{j1}^{i}+\delta _{j}^{i}\left[ \varkappa
_{111}-{\dfrac{dh_{11}}{dt}}\right] }, 
\]%
\[
\begin{array}{cc}
h_{11}L_{(1)(j)k}^{(i)(1)}=h_{11}L_{jk}^{i}, & 
h_{11}C_{(1)(j)(k)}^{(i)(1)(1)}=h_{11}C_{j(k)}^{i(1)},%
\end{array}%
\]%
where $\varkappa _{111}=\varkappa _{11}^{1}h_{11}$ represent the Christoffel
symbols of first kind attached to the Riemannian metric $h_{11}(t)$.
Contracting the above relations with the inverse $h^{11}=1/h_{11}$, we
obtain the last three identities from (\ref{h-normal-relations}).
\end{proof}

\begin{remark}
The Theorem \ref{TH-h-normal} implies that an $h$-normal $\Gamma $-linear
connection $\nabla $ is determined by \textbf{four} \underline{effective}
local components (instead of \textbf{two} effective \nolinebreak local
components for an $N$-linear connection in the Miron-Anastasiei case \cite%
{Mir-An} or \cite{Mir-Shimada}, pp. 21), namely%
\begin{equation}
\nabla \Gamma =\left( \varkappa
_{11}^{1},G_{i1}^{k},L_{ij}^{k},C_{i(j)}^{k(1)}\right) .
\label{h-normal-gamma-linear-connect}
\end{equation}%
The other five components of $\nabla $ cancel or depend by the above four
components, via the formulas (\ref{h-normal-relations}).
\end{remark}

\begin{example}
The canonical Berwald $\mathring{\Gamma}$-linear connection associated to
the pair of Riemannian metrics $(h_{11}(t),\varphi _{ij}(x))$ is an $h$%
-normal $\mathring{\Gamma}$-linear connection, defined by the local
components $B\mathring{\Gamma}=\left( \varkappa _{11}^{1},0,\gamma
_{ij}^{k},0\right) .$
\end{example}

The study of adapted components of the torsion tensor $\mathbf{T}$ and
curvature tensor $\mathbf{R}$ of an arbitrary $\Gamma $-linear connection $%
\nabla $ on $E=J^{1}(\mathbb{R},M)$ was completely done in the paper \cite%
{Neagu-Stoica}. In that paper, we proved that the torsion tensor $\mathbf{T}$
is determined by \textit{ten} effective local d-tensors, while the curvature
tensor $\mathbf{R}$ is determined by \textit{fifteen} effective local
d-tensors. In the sequel, we study the adapted components of the torsion and
curvature tensors for an $h$-normal $\Gamma $-linear connection $\nabla $
given by (\ref{h-normal-gamma-linear-connect}) and (\ref{h-normal-relations}%
).

\begin{theorem}
The torsion tensor $\mathbf{T}$ of an $h$-normal $\Gamma $-linear connection 
$\nabla $ on $E$ is determined by the following \textbf{eight} adapted local
d-tensors (instead of \textbf{ten} in the general case \cite{Neagu-Stoica},
pp. 12): 
\begin{equation}
\begin{tabular}{|c|c|c|c|}
\hline
& $h_{\mathbb{R}}$ & $h_{M}$ & $v$ \\ \hline
$h_{\mathbb{R}}h_{\mathbb{R}}$ & $0$ & $0$ & $0$ \\ \hline
$h_{M}h_{\mathbb{R}}$ & $0$ & $T_{1j}^{r}$ & $R_{(1)1j}^{(r)}$ \\ \hline
$h_{M}h_{M}$ & $0$ & $T_{ij}^{r}$ & $R_{(1)ij}^{(r)}$ \\ \hline
$vh_{\mathbb{R}}$ & $0$ & $0$ & $P_{(1)1(j)}^{(r)\;\;(1)}$ \\ \hline
$vh_{M}$ & $0$ & $P_{i(j)}^{r(1)}$ & $P_{(1)i(j)}^{(r)\;(1)}$ \\ \hline
$vv$ & $0$ & $0$ & $S_{(1)(i)(j)}^{(r)(1)(1)}$ \\ \hline
\end{tabular}
\label{Table-Tors-h-Normal}
\end{equation}%
where\medskip

$\mathbf{1}$\textbf{.} $T_{1j}^{r}=-G_{j1}^{r},\medskip $

$\mathbf{2}$\textbf{.} ${R_{(1)1j}^{(r)}={\dfrac{\delta M_{(1)1}^{(r)}}{%
\delta x^{j}}}-{\dfrac{\delta N_{(1)j}^{(r)}}{\delta t}},}\medskip $

$\mathbf{3}$\textbf{.} $T_{ij}^{r}=L_{ij}^{r}-L_{ji}^{r},\medskip $

$\mathbf{4}$\textbf{.} ${R_{(1)ij}^{(r)}={\dfrac{\delta N_{(1)i}^{(r)}}{%
\delta x^{j}}}-{\dfrac{\delta N_{(1)j}^{(r)}}{\delta x^{i}}},}\medskip $

$\mathbf{5}$\textbf{.} ${P_{(1)1(j)}^{(r)\;\;(1)}={\dfrac{\partial
M_{(1)1}^{(r)}}{\partial y_{1}^{j}}}-G_{j1}^{r}+\delta _{j}^{r}\varkappa
_{11}^{1},}\medskip $

$\mathbf{6}$\textbf{.} $P_{i(j)}^{r(1)}=C_{i(j)}^{r(1)}{,}\medskip $

$\mathbf{7}$\textbf{.} $P_{(1)i(j)}^{(r)\;(1)}={\dfrac{\partial
N_{(1)i}^{(r)}}{\partial y_{1}^{j}}}-L_{ji}^{r},\medskip $

$\mathbf{8}$\textbf{.} ${%
S_{(1)(i)(j)}^{(r)(1)(1)}=C_{i(j)}^{r(1)}-C_{j(i)}^{r(1)}}.$
\end{theorem}

\begin{proof}
Particularizing the general local expressions from \cite{Neagu-Stoica}
(which give those ten components of the torsion tensor of a $\Gamma $-linear
connection $\nabla $, in the large) for an $h$-normal $\Gamma $-linear
connection $\nabla $, we deduce that the adapted local components $\bar{T}%
_{1j}^{1}$ and $\bar{P}_{1(j)}^{1(1)}$ vanish, while the other eight ones
from the Table \ref{Table-Tors-h-Normal} are expressed by the preceding
formulas.
\end{proof}

\begin{remark}
The torsion of an $N$-linear connection in the Miron-Anastasiei case is
characterized only by \textbf{five} effective adapted components (please see 
\cite{Mir-An} or \cite{Mir-Shimada}, pp. 24).
\end{remark}

\begin{remark}
For the Berwald $\mathring{\Gamma}$-linear connection $B\mathring{\Gamma}$
associated to the Riemannian metrics $h_{11}(t)$ and $\varphi _{ij}(x)$, all
adapted local torsion d-tensors vanish, except%
\[
R_{(1)ij}^{(k)}=\mathfrak{R}_{mij}^{k}y_{1}^{m}, 
\]%
where $\mathfrak{R}_{mij}^{k}(x)$ are the classical local curvature
components of the Riemannian metric $\varphi _{ij}(x)$.
\end{remark}

The expressions of the local curvature d-tensors of an arbitrary $\Gamma $%
-linear connection, together with the particular properties of an $h$-normal 
$\Gamma $-linear connection, imply a considerable reduction (from \textit{%
fifteen} to \textit{five}) of the \underline{effective} local curvature
d-tensors that characterize an $h$-normal $\Gamma $-linear connection. In
other words, we have

\begin{theorem}
The curvature tensor $\mathbf{R}$ of an $h$-normal $\Gamma $-linear
connection $\nabla $ on $E$ is characterized by \textbf{five} \underline{%
effective} local curvature d-tensors (instead of \textbf{fifteen} in the
general case \cite{Neagu-Stoica}, pp. 14):%
\begin{equation}
\begin{tabular}{|c|c|c|c|}
\hline
& $h_{\mathbb{R}}$ & $h_{M}$ & $v$ \\ \hline
$h_{\mathbb{R}}h_{\mathbb{R}}$ & $0$ & $0$ & $0$ \\ \hline
$h_{M}h_{\mathbb{R}}$ & $0$ & $R_{i1k}^{l}$ & $%
R_{(1)(i)1k}^{(l)(1)}=R_{i1k}^{l}$ \\ \hline
$h_{M}h_{M}$ & $0$ & $R_{ijk}^{l}$ & $R_{(1)(i)jk}^{(l)(1)}=R_{ijk}^{l}$ \\ 
\hline
$vh_{\mathbb{R}}$ & $0$ & $P_{i1(k)}^{l\;\;(1)}$ & $P_{(1)(i)1(k)}^{(l)(1)\;%
\;(1)}=P_{i1(k)}^{l\;\;(1)}$ \\ \hline
$vh_{M}$ & $0$ & $P_{ij(k)}^{l\;(1)}$ & $P_{(1)(i)j(k)}^{(l)(1)%
\;(1)}=P_{ij(k)}^{l\;(1)}$ \\ \hline
$vv$ & $0$ & $S_{i(j)(k)}^{l(1)(1)}$ & $%
S_{(1)(i)(j)(k)}^{(l)(1)(1)(1)}=S_{i(j)(k)}^{l(1)(1)}$ \\ \hline
\end{tabular}
\label{Table-Curv-h-Normal}
\end{equation}%
\medskip where\medskip

$\mathbf{1}$\textbf{.} ${R_{i1k}^{l}={\dfrac{\delta G_{i1}^{l}}{\delta x^{k}}%
}-{\dfrac{\delta L_{ik}^{l}}{\delta t}}%
+G_{i1}^{r}L_{rk}^{l}-L_{ik}^{r}G_{r1}^{l}+C_{i(r)}^{l(1)}R_{(1)1k}^{(r)},}$%
\medskip

$\mathbf{2}$\textbf{.} ${R_{ijk}^{l}={\dfrac{\delta L_{ij}^{l}}{\delta x^{k}}%
}-{\dfrac{\delta L_{ik}^{l}}{\delta x^{j}}}%
+L_{ij}^{r}L_{rk}^{l}-L_{ik}^{r}L_{rj}^{l}+C_{i(r)}^{l(1)}R_{(1)jk}^{(r)},}$%
\medskip

$\mathbf{3}$\textbf{.} ${P_{i1(k)}^{l\;\;(1)}={\dfrac{\partial G_{i1}^{l}}{%
\partial y_{1}^{k}}}-C_{i(k)/1}^{l(1)}+C_{i(r)}^{l(1)}P_{(1)1(k)}^{(r)\;%
\;(1)},}$\medskip

$\mathbf{4}$\textbf{.} ${P_{ij(k)}^{l\;(1)}={\dfrac{\partial L_{ij}^{l}}{%
\partial y_{1}^{k}}}-C_{i(k)|j}^{l(1)}+C_{i(r)}^{l(1)}P_{(1)j(k)}^{(r)\;(1)},%
}$\medskip

$\mathbf{5}$\textbf{.} ${S_{i(j)(k)}^{l(1)(1)}={\dfrac{\partial
C_{i(j)}^{l(1)}}{\partial y_{1}^{k}}}-{\dfrac{\partial C_{i(k)}^{l(1)}}{%
\partial y_{1}^{j}}}%
+C_{i(j)}^{r(1)}C_{r(k)}^{l(1)}-C_{i(k)}^{r(1)}C_{r(j)}^{l(1)}.}$
\end{theorem}

\begin{proof}
The general formulas that express those fifteen local curvature d-tensors of
an arbitrary $\Gamma $-linear connection \cite{Neagu-Stoica}, applied to the
particular case of an $h$-normal $\Gamma $-linear connection $\nabla $ on $E$%
, imply the preceding formulas and the relations from the Table \ref%
{Table-Curv-h-Normal}.
\end{proof}

\begin{remark}
The curvature of an $N$-linear connection in the Miron-Anastasiei case is
characterized only by \textbf{three} effective adapted components (please
see \cite{Mir-An} or \cite{Mir-Shimada}, pp. 25).
\end{remark}

\begin{remark}
For the Berwald $\mathring{\Gamma}$-linear connection $B\mathring{\Gamma}$
associated to the Riemannian metrics $h_{11}(t)$ and $\varphi _{ij}(x)$, all
local curvature d-tensors vanish, except 
\[
R_{(1)(i)jk}^{(l)(1)}=R_{ijk}^{l}=\mathfrak{R}_{ijk}^{l}, 
\]%
where $\mathfrak{R}_{ijk}^{l}(x)$ are the local curvature tensors of the
Riemannian metric $\varphi _{ij}(x)$.
\end{remark}

\section{d-Connections of Cartan type. Local Ricci and Bianchi identities}

\hspace{5mm}Because of the reduced number and the simplified form of the
local torsion and curvature d-tensors of an $h$-normal $\Gamma $-linear
connection $\nabla $ on the 1-jet space $E$, the number of attached \textit{%
local Ricci and Bianchi identities} considerably simplifies. A substantial
reduction of these identities obtains considering the more particular case
of an $h$-normal $\Gamma $-linear connection of \textit{Cartan type.}

\begin{definition}
An $h$-normal $\Gamma $-linear connection on $E=J^{1}(\mathbb{R},M)$, whose
local components%
\[
\nabla \Gamma =\left( \varkappa
_{11}^{1},G_{i1}^{k},L_{ij}^{k},C_{i(j)}^{k(1)}\right) 
\]%
verify the supplementary conditions $L_{ij}^{k}=L_{ji}^{k}$ and $%
C_{i(j)}^{k(1)}=C_{j(i)}^{k(1)},$ is called an $h$\textbf{-normal }$\Gamma $%
\textbf{-linear connection of Cartan type}.
\end{definition}

\begin{remark}
In the particular case of an $h$-normal $\Gamma $-linear connection of
Cartan type, the conditions $L_{ij}^{k}=L_{ji}^{k}$ and $%
C_{i(j)}^{k(1)}=C_{j(i)}^{k(1)}$ imply the torsion equalities%
\[
\begin{array}{cc}
T_{ij}^{k}=0, & {S_{(1)(i)(j)}^{(k)(1)(1)}=0}.%
\end{array}%
\]
\end{remark}

Rewriting the local Ricci identities of a $\Gamma $-linear connection $%
\nabla $ (described in the large in \cite{Neagu-Stoica}, pp. 15), for the
particular case of an $h$-normal $\Gamma $-linear connection of Cartan type,
we find a simplified form of these identities. Consequently, we obtain

\begin{theorem}
The following \textbf{local Ricci identities} for an $h$-normal $\Gamma $%
-linear connection of Cartan type are true:$\bigskip$\newline
$(h_{\mathbb{R}})\mbox{\hspace{3mm}}\left\{ 
\begin{array}{l}
\medskip
X_{/1|k}^{1}-X_{|k/1}^{1}=-X_{|r}^{1}T_{1k}^{r}-X^{1}|_{(r)}^{(1)}R_{(1)1k}^{(r)}
\\ 
\medskip X_{|j|k}^{1}-X_{|k|j}^{1}=-X^{1}|_{(r)}^{(1)}R_{(1)jk}^{(r)} \\ 
\medskip
X_{/1}^{1}|_{(k)}^{(1)}-X^{1}|_{(k)/1}^{(1)}=-X^{1}|_{(r)}^{(1)}P_{(1)1(k)}^{(r)\;\;(1)}
\\ 
\medskip
X_{|j}^{1}|_{(k)}^{(1)}-X^{1}|_{(k)|j}^{(1)}=-X_{|r}^{1}C_{j(k)}^{r(1)}-X^{1}|_{(r)}^{(1)}P_{(1)j(k)}^{(r)\;(1)}
\\ 
X^{1}|_{(j)}^{(1)}|_{(k)}^{(1)}-X^{1}|_{(k)}^{(1)}|_{(j)}^{(1)}=0,%
\end{array}%
\right. \bigskip$\newline
$(h_{M})\mbox{\hspace{2mm}}\left\{ 
\begin{array}{l}
\medskip
X_{/1|k}^{i}-X_{|k/1}^{i}=X^{r}R_{r1k}^{i}-X_{|r}^{i}T_{1k}^{r}-X^{i}|_{(r)}^{(1)}R_{(1)1k}^{(r)}
\\ 
\medskip
X_{|j|k}^{i}-X_{|k|j}^{i}=X^{r}R_{rjk}^{i}-X^{i}|_{(r)}^{(1)}R_{(1)jk}^{(r)}
\\ 
\medskip
X_{/1}^{i}|_{(k)}^{(1)}-X^{i}|_{(k)/1}^{(1)}=X^{r}P_{r1(k)}^{i\;%
\;(1)}-X^{i}|_{(r)}^{(1)}P_{(1)1(k)}^{(r)\;\;(1)} \\ 
\medskip X_{|j}^{i}|_{(k)}^{(1)}-X^{i}|_{(k)|j}^{(1)}=X^{r}P_{rj(k)}^{i\;\
(1)}-X_{|r}^{i}C_{j(k)}^{r(1)}-X^{i}|_{(r)}^{(1)}P_{(1)j(k)}^{(r)\;(1)} \\ 
X^{i}|_{(j)}^{(1)}|_{(k)}^{(1)}-X^{i}|_{(k)}^{(1)}|_{(j)}^{(1)}=X^{r}S_{r(j)(k)}^{i(1)(1)},%
\end{array}%
\right. \bigskip$\newline
$\bigskip(v)\mbox{\hspace{6mm}}\left\{ 
\begin{array}{l}
\medskip
X_{(1)/1|k}^{(i)}-X_{(1)|k/1}^{(i)}=X_{(1)}^{(r)}R_{r1k}^{i}-X_{(1)|r}^{(i)}T_{1k}^{r}-X_{(1)}^{(i)}|_{(r)}^{(1)}R_{(1)1k}^{(r)}
\\ 
\medskip
X_{(1)|j|k}^{(i)}-X_{(1)|k|j}^{(i)}=X_{(1)}^{(r)}R_{rjk}^{i}-X_{(1)}^{(i)}|_{(r)}^{(1)}R_{(1)jk}^{(r)}
\\ 
\medskip
X_{(1)/1}^{(i)}|_{(k)}^{(1)}-X_{(1)}^{(i)}|_{(k)/1}^{(1)}=X_{(1)}^{(r)}P_{r1(k)}^{i\;\;(1)}-X_{(1)}^{(i)}|_{(r)}^{(1)}P_{(1)1(k)}^{(r)\;\;(1)}
\\ 
\medskip
X_{(1)|j}^{(i)}|_{(k)}^{(1)}-X_{(1)}^{(i)}|_{(k)|j}^{(1)}=X_{(1)}^{(r)}P_{rj(k)}^{i\;\ (1)}-X_{(1)|r}^{(i)}C_{j(k)}^{r(1)}-X_{(1)}^{(i)}|_{(r)}^{(1)}P_{(1)j(k)}^{(r)\;(1)}
\\ 
X_{(1)}^{(i)}|_{(j)}^{(1)}|_{(k)}^{(1)}-X_{(1)}^{(i)}|_{(k)}^{(1)}|_{(j)}^{(1)}=X_{(1)}^{(r)}S_{r(j)(k)}^{i(1)(1)},%
\end{array}%
\right. $\newline
where%
\[
{X=X^{1}{\frac{\delta }{\delta t}}+X^{i}{\frac{\delta }{\delta x^{i}}}%
+X_{(1)}^{(i)}{\frac{\partial }{\partial y_{1}^{i}}}} 
\]%
is an arbitrary distinguished vector field on the 1-jet space $E=J^{1}(%
\mathbb{R},M)$.
\end{theorem}

In what follows, let us consider the \textit{canonical jet Liouville
d-vector field}%
\[
\mathbf{C}{=\mathbf{C}_{(1)}^{(i)}{\frac{\partial }{\partial y_{1}^{i}}=}}%
y_{1}^{i}{{\frac{\partial }{\partial y_{1}^{i}}}}, 
\]%
and let us construct the \textit{nonmetrical deflection d-tensors}%
\[
\bar{D}_{(1)1}^{(i)}=\mathbf{C}_{(1)/1}^{(i)},\quad D_{(1)j}^{(i)}=\mathbf{C}%
_{(1)|j}^{(i)},\quad d_{(1)(j)}^{(i)(1)}=\mathbf{C}%
_{(1)}^{(i)}|_{(j)}^{(1)}. 
\]%
Then, a direct calculation leads to

\begin{proposition}
The nonmetrical deflection d-tensors attached to the $h$-normal $\Gamma $%
-linear connection $\nabla $ given by (\ref{h-normal-gamma-linear-connect})
have the expressions%
\begin{equation}
\begin{array}{c}
\medskip \bar{D}_{(1)1}^{(i)}=-M_{(1)1}^{(i)}+G_{r1}^{i}y_{1}^{r}-\varkappa
_{11}^{1}y_{1}^{i},\qquad D_{(1)j}^{(i)}=-N_{(1)j}^{(i)}+L_{rj}^{i}y_{1}^{r},
\\ 
d_{(1)(j)}^{(i)(1)}=\delta _{j}^{i}+C_{r(j)}^{i(1)}y_{1}^{r}.%
\end{array}
\label{defl-for-h-normal}
\end{equation}
\end{proposition}

Applying now the preceding $(v)-$ set of local Ricci identities (associated
to an $h$-normal $\Gamma $-linear connection of Cartan type) to the
components of the canonical jet Liouville d-vector field, we find

\begin{theorem}
The following \textbf{five }identities of the nonmetrical deflection
d-tensors associated to an $h$-normal $\Gamma $-linear connection of Cartan
type (instead of \textbf{three} in the Miron-Anastasiei's case \cite{Mir-An}
or \cite{Mir-Shimada}, pp. 80) are true: 
\begin{equation}
\left\{ 
\begin{array}{l}
\medskip \bar{D}%
_{(1)1|k}^{(i)}-D_{(1)k/1}^{(i)}=y_{1}^{r}R_{r1k}^{i}-D_{(1)r}^{(i)}T_{1k}^{r}-d_{(1)(r)}^{(i)(1)}R_{(1)1k}^{(r)}
\\ 
\medskip
D_{(1)j|k}^{(i)}-D_{(1)k|j}^{(i)}=y_{1}^{r}R_{rjk}^{i}-d_{(1)(r)}^{(i)(1)}R_{(1)jk}^{(r)}
\\ 
\medskip \bar{D}%
_{(1)1}^{(i)}|_{(k)}^{(1)}-d_{(1)(k)/1}^{(i)(1)}=y_{1}^{r}P_{r1(k)}^{i\;%
\;(1)}-d_{(1)(r)}^{(i)(1)}P_{(1)1(k)}^{(r)\;\;(1)} \\ 
\medskip
D_{(1)j}^{(i)}|_{(k)}^{(1)}-d_{(1)(k)|j}^{(i)(1)}=y_{1}^{r}P_{rj(k)}^{i\;\
(1)}-D_{(1)r}^{(i)}C_{j(k)}^{r(1)}-d_{(1)(r)}^{(i)(1)}P_{(1)j(k)}^{(r)\;(1)}
\\ 
d_{(1)(j)}^{(i)(1)}|_{(k)}^{(1)}-d_{(1)(k)}^{(i)(1)}|_{(j)}^{(1)}=y_{1}^{r}S_{r(j)(k)}^{i(1)(1)}.%
\end{array}%
\right.  \label{defl-ID-h-normal-Cartan}
\end{equation}
\end{theorem}

\begin{remark}
The identities (\ref{defl-ID-h-normal-Cartan}) are used in the description
of \textbf{generalized Maxwell equations} that govern the \textbf{%
electromagnetic 2-form} \cite{Neagu-Rheon} produced by a relativistic time
dependent Lagrangian on the 1-jet space $E=J^{1}(\mathbb{R},M)$.
\end{remark}

The using of $h$-normal $\Gamma $-linear connections of Cartan type in the
study of differential geometry of the 1-jet vector bundle $E=J^{1}(\mathbb{R}%
,M)$ is also convenient because the number and the form of the local Bianchi
identities associated to such connections are considerably simplified. In
fact, we have:

\begin{theorem}
The following \textbf{nineteen} effective \textbf{local Bianchi identities}
for the $h$-normal $\Gamma $-linear connections of Cartan type $\nabla $
given by (\ref{h-normal-gamma-linear-connect}) are true:\medskip

$\medskip \mathbf{1}$\textbf{.} $\mathcal{A}_{\{j,k\}}\left\{
R_{j1k}^{l}+T_{1j|k}^{l}+R_{(1)1j}^{(r)}C_{k(r)}^{l(1)}\right\} =0,$

$\medskip \mathbf{2}^{\text{\textbf{*}}}$\textbf{.} $\sum_{\{i,j,k\}}\left\{
R_{ijk}^{l}-R_{(1)ij}^{(r)}C_{k(r)}^{l(1)}\right\} =0,$

$\medskip \mathbf{3}$\textbf{.} $\mathcal{A}_{\{j,k\}}\left\{
R_{(1)1j|k}^{(l)}+T_{1j}^{r}R_{(1)kr}^{(l)}+R_{(1)1j}^{(r)}P_{(1)k(r)}^{(l)%
\;\ (1)}\right\} =$

$\medskip \mbox{\hspace{10mm}}%
=-R_{(1)jk/1}^{(l)}-R_{(1)jk}^{(r)}P_{(1)1(r)}^{(l)\;\;(1)},$

$\medskip \mathbf{4}^{\text{\textbf{*}}}$\textbf{.} $\sum_{\{i,j,k\}}\left\{
R_{(1)ij|k}^{(l)}+R_{(1)ij}^{(r)}P_{(1)k(r)}^{(l)\;\text{\ }(1)}\right\} =0,$

$\medskip \mathbf{5}$\textbf{.} $%
T_{1k}^{l}|_{(p)}^{(1)}-C_{r(p)}^{l(1)}T_{1k}^{r}+P_{k1(p)}^{l\;%
\;(1)}+C_{k(p)/1}^{l(1)}+C_{k(p)}^{r(1)}T_{1r}^{l}-C_{k(r)}^{l(1)}P_{(1)1(p)}^{(r)\;\;(1)}=0, 
$

$\medskip \mathbf{6}^{\text{\textbf{*}}}$\textbf{.} $\mathcal{A}%
_{\{j,k\}}\left\{ C_{j(p)|k}^{l(1)}+C_{k(r)}^{l(1)}P_{(1)j(p)}^{(r)\;\
(1)}+P_{jk(p)}^{l\;\ (1)}\right\} =0,$

$\medskip \mathbf{7}$\textbf{.} $P_{(1)1(p)|k}^{(l)\;%
\;(1)}-P_{(1)k(p)/1}^{(l)\;\;(1)}+P_{(1)k(r)}^{(l)\;\;(1)}P_{(1)1(p)}^{(r)\;%
\;(1)}-P_{(1)1(r)}^{(l)\;\;(1)}P_{(1)k(p)}^{(r)\;\;(1)}=$

$\medskip \mbox{\hspace{10mm}}%
=R_{(1)1k}^{(l)}|_{(p)}^{(1)}-R_{p1k}^{l}+R_{(1)1r}^{(l)}C_{k(p)}^{r(1)}-T_{1k}^{r}P_{(1)r(p)}^{(l)\;\ (1)}, 
$

$\medskip \mathbf{8}^{\text{\textbf{*}}}$\textbf{.} $\mathcal{A}%
_{\{j,k\}}\left\{
R_{(1)jr}^{(l)}C_{k(p)}^{r(1)}+P_{(1)j(r)}^{(l)\;\;(1)}P_{(1)k(p)}^{(r)\;\
(1)}+P_{(1)k(p)|j}^{(l)\;\;(1)}\right\}
=R_{pjk}^{l}-R_{(1)jk}^{(l)}|_{(p)}^{(1)},$

$\mathbf{9}^{\text{\textbf{*}}}$\textbf{.} $\medskip \mathcal{A}_{\left\{
j,k\right\} }\left\{
C_{i(j)}^{l(1)}|_{(k)}^{(1)}+C_{i(k)}^{r(1)}C_{r(j)}^{l(1)}\right\} =$

$\medskip\mbox{\hspace{10mm}}=\mathcal{A}_{\left\{ j,k\right\} }\left\{ 
\dfrac{C_{i(j)}^{l(1)}}{\partial y_{1}^{k}}+C_{i(j)}^{r(1)}C_{r(k)}^{l(1)}%
\right\} =S_{i(j)(k)}^{l(1)(1)},$

$\medskip \mathbf{10}$\textbf{.} $\mathcal{A}_{\left\{ j,k\right\} }\left\{
P_{(1)1(j)}^{(l)\;\;(1)}|_{(k)}^{(1)}+P_{j1(k)}^{l\;\;(1)}\right\} =0,$

$\medskip \mathbf{11}^{\text{\textbf{*}}}$\textbf{.} $\mathcal{A}_{\left\{
j,k\right\} }\left\{ P_{ji(k)}^{l\;\ (1)}+P_{(1)r(j)}^{(l)\;\
(1)}C_{i(k)}^{r(1)}-P_{(1)i(k)}^{(l)\;(1)}|_{(j)}^{(1)}\right\} =0,$

$\medskip \mathbf{12}^{\text{\textbf{*}}}$\textbf{.} $\sum_{\left\{
i,j,k\right\} }S_{i(j)(k)}^{l(1)(1)}=0,$

$\medskip \mathbf{13}$\textbf{.} $\mathcal{A}_{\{j,k\}}\left\{
R_{p1j|k}^{l}+T_{1j}^{r}R_{pkr}^{l}+R_{(1)1j}^{(r)}P_{pk(r)}^{l\;\;(1)}%
\right\} =-R_{pjk/1}^{l}-R_{(1)jk}^{(r)}P_{p1(r)}^{l\;\text{\ }(1)},$

$\medskip \mathbf{14}^{\text{\textbf{*}}}$\textbf{.} $\sum_{\{i,j,k\}}\left%
\{ R_{pij|k}^{l}+R_{(1)ij}^{(r)}P_{pk(r)}^{l\;\ (1)}\right\} =0,$

$\medskip \mathbf{15}$\textbf{.} $P_{i1(p)|k}^{l\;\;(1)}-P_{ik(p)/1}^{l\;%
\;(1)}+P_{(1)1(p)}^{(r)\;\;(1)}P_{ik(r)}^{l\;\;(1)}-P_{(1)k(p)}^{(r)\;%
\;(1)}P_{i1(r)}^{l\;\;(1)}=$

$\medskip \mbox{\hspace{10mm}}%
=R_{i1k}^{l}|_{(p)}^{(1)}+R_{(1)1k}^{(r)}S_{i(p)(r)}^{l(1)(1)}+C_{k(p)}^{r(1)}R_{i1r}^{l}-T_{1k}^{r}P_{ir(p)}^{l\;\;(1)}, 
$

$\medskip \mathbf{16}^{\text{\textbf{*}}}$\textbf{.} $\mathcal{A}%
_{\{j,k\}}\left\{
R_{ijr}^{l}C_{k(p)}^{r(1)}+P_{ij(r)}^{l\;\;(1)}P_{(1)k(p)}^{(r)\;%
\;(1)}+P_{ik(p)|j}^{l\;\;(1)}\right\} =$

$\medskip \mbox{\hspace{10mm}}%
=-S_{i(p)(r)}^{l(1)(1)}R_{(1)jk}^{(r)}-R_{ijk}^{l}|_{(p)}^{(1)},$

$\medskip \mathbf{17}$\textbf{.} $\mathcal{A}_{\left\{ j,k\right\} }\left\{
P_{p1(j)}^{l\;\;(1)}|_{(k)}^{(1)}+P_{(1)1(j)}^{(r)\;%
\;(1)}S_{p(k)(r)}^{l(1)(1)}\right\} =-S_{p(j)(k)/1}^{l(1)(1)},$

$\medskip \mathbf{18}^{\text{\textbf{*}}}$\textbf{.} $\mathcal{A}_{\left\{
j,k\right\} }\left\{
P_{pr(j)}^{l\;\;(1)}C_{i(k)}^{r(1)}-S_{p(j)(r)}^{l(1)(1)}P_{(1)i(k)}^{(r)\;%
\;(1)}-P_{pi(k)}^{l\;\;(1)}|_{(j)}^{(1)}\right\} =-S_{p(j)(k)|i}^{l(1)(1)},$

$\medskip \mathbf{19}^{\text{\textbf{*}}}$\textbf{.}$\mathbf{\;}%
\sum_{\left\{ i,j,k\right\} }S_{p(i)(j)}^{l(1)(1)}|_{(k)}^{(1)}=0,$ where $%
\sum_{\{i,j,k\}}$ represents a cyclic sum and $\mathcal{A}_{\{i,j\}}$
represents an alternate sum.
\end{theorem}

\begin{proof}
Let ${(X_{A})=\left( \delta /\delta t,\delta /\delta x^{i},\partial
/\partial y_{1}^{i}\right) }$ be the adapted basis of vector fields produced
by the nonlinear connection (\ref{gen_nlc}). Let $\nabla $ be the $h$-normal 
$\Gamma $-linear connection of Cartan type given by (\ref%
{h-normal-gamma-linear-connect}). For the linear connection $\nabla $ the
following general local Bianchi identities are true \cite{Mir-An}, \cite%
{Neagu-Bianchi}:%
\[
\sum_{\{A,B,C\}}\left\{ \mathbf{R}_{ABC}^{F}-\mathbf{T}_{AB:C}^{F}-\mathbf{T}%
_{AB}^{G}\mathbf{T}_{CG}^{F}\right\} =0,
\]%
\[
\sum_{\{A,B,C\}}\left\{ \mathbf{R}_{DAB:C}^{F}+\mathbf{T}_{AB}^{G}\mathbf{R}%
_{DCG}^{F}\right\} =0,
\]%
where $\mathbf{R}(X_{A},X_{B})X_{C}=\mathbf{R}_{CBA}^{D}X_{D}$, $\mathbf{T}%
(X_{A},X_{B})=\mathbf{T}_{BA}^{D}X_{D}$ and \textquotedblright $_{:A}$%
\textquotedblright\ represents one from the local covariant derivatives
\textquotedblright $_{/1}$\textquotedblright , \textquotedblright $_{|i}$%
\textquotedblright\ or \textquotedblright $|_{(i)}^{(1)}$\textquotedblright
. Obviously, the components $\mathbf{T}_{AB}^{C}$ and $\mathbf{R}_{ABC}^{D}$
are the adapted components of the torsion and curvature tensors associated
to the linear connection $\nabla \Gamma $. These components are expressed in
the Tables \ref{Table-Tors-h-Normal} and \ref{Table-Curv-h-Normal}. Then,
replacing $A,B,C,\ldots $ with indices of type 
\[
\left\{ 1,i,{\QATOP{(1)}{(i)}}\right\} ,
\]%
by laborious local computations, we obtain the required Bianchi identities.
\end{proof}

\begin{remark}
The above \textbf{eleven "star"-Bianchi identities }are exactly those eleven
Bianchi identities that characterize the canonical metrical Cartan %
co\-nnec\-tion of a Finsler space (please see \cite{Mir-Shimada}, pp.48).
\end{remark}

\begin{remark}
The importance of preceding Bianchi identities for an $h$-nor\-mal $\Gamma $%
-linear connection of Cartan type $\nabla $ on the 1-jet space $J^{1}(%
\mathbb{R},M)$ comes from their using in the local description of the 
\textbf{generalized Maxwell equations} (please see \cite{Neagu-Rheon}, pp.
161) that characterize an electromagnetic field in the background of
relativistic non-autonomous Lagrange geometry.
\end{remark}

\textbf{Author's address:}\medskip

Mircea N{\scriptsize EAGU}

University Transilvania of Bra\c{s}ov, Faculty of Mathematics and Informatics

Department of Algebra, Geometry and Differential Equations

B-dul Iuliu Maniu, Nr. 50, BV 500091, Bra\c{s}ov, Romania.

\textit{E-mail}: mircea.neagu@unitbv.ro

\textit{Website}: http://www.2collab.com/user:mirceaneagu

\end{document}